\documentclass[12pt]{amsart}
\usepackage{amssymb,amsmath}
\usepackage{pgfplots}
\pgfplotsset{compat=newest}
\usepackage[english]{babel}
\usepackage{url}

\makeatletter
\theoremstyle{plain}
 \newtheorem{thm}{Theorem}[section]
\newtheorem{thm*}{Theorem}
 \newtheorem{lem}[thm]{Lemma}
 \newtheorem{prop}[thm]{Proposition}
 
 \numberwithin{equation}{section} %% Comment out for sequentially-numbered
\numberwithin{figure}{section} %% Comment out for sequentially-numbered
 \theoremstyle{plain}
 \theoremstyle{definition}
 
 \newtheorem{rem}[thm]{Remark}
\newcommand{\A}{{{\mathbb A}}}
\newcommand{\bH}{{{\bf H}}}
\newcommand{\calA}{{{\mathcal A}}}

\newcommand{\fH}{{{\mathfrak H}}}

\newcommand{\calE}{{{\mathcal E}}}
\newcommand{\fp}{{{\mathfrak p}}}

\newcommand{\C}{{{\mathbb C}}}

\newcommand{\R}{{{\mathbb R}}}

\newcommand{\calET}{{{\mathcal E}{\mathcal T}}}
\newcommand{\fET}{{{\mathfrak E}{\mathfrak T}}}
\newcommand{\X}{{{\mathbb X}}}
\newcommand{\fX}{{{\mathfrak X}}}

\newcommand{\bp}{{{\bf p}}}

\newcommand{\binfty}{{{\bf \infty}}}

\makeatother

\date{\today\\
2010 \emph{Mathematics Subject Classifications.} 32M99, 51F99.\\
\emph{Key words.} Heisenberg group, Kor\'anyi metric, equidistant triples}
\begin{document}

\title[Configuration space of equidistant triples]{The configuration space of equidistant triples in the Heisenberg group}

\author{%V. Chousionis \& 
Ioannis D. Platis %\& V. Schroeder
}

\begin{abstract}
We prove that the configuration space of equidistant triples on the Heisenberg group equipped with the Kor\'anyi metric, is isomorphic to a hypersurface of $\R^3$.
\end{abstract}

\address{Department of Mathematics and Applied Mathematics, University Campus, University of Crete, GR-70013 Voutes, Heraklion Crete, Greece}
\email{jplatis@math.uoc.gr}

\maketitle

\section{Introduction}

Let $\fH$ be the first Heisenberg group equipped with the Kor\'anyi distance $d$. An equidistant triple is a triple of points $P=(p_1,p_2,p_3)\in\fH$ such that
$$
d(p_1,p_2)=d(p_2,p_3)=d(p_3,p_1).
$$ 
Denote by $\calET$ the space of equidistant triples in $\fH$. Then the similarity group $${G}={\rm Sim}(\fH, d)=\fH\times\R\times \R^*_+,$$
acts diagonally on $\calET$:
$$
\left(g,(p_1,p_2,p_3)\right)\mapsto\left(g(p_1),g(p_2),g(p_3)\right).
$$
%The group $G$ comprises left-translations, rotations, dilations and involution $j:(z,t\mapsto(\overline{z},-t)$. 
We denote by $\fET$ the quotient of this action; this is the {\it configuration space of $G$-equivalent equidistant triples in $\fH$}. In this paper we are dealing with the problem of parametrising $\fET$. The problem is addressed and solved in a different manner in \cite{CT} (see Proposition 4.6 there). We prove here the following theorem: 

\begin{thm}\label{thm:basic}
The configuration space $\fET$ of $G$-equivalent equidistant triples is in bijection with the hypersurface $\calE$ of $\R^3$ which is defined by
$$
\calE=\left\{(a,b,c)\in[-2\pi/3,2\pi/3]^3\;|\;\cos a+\cos b+\cos c=\frac{3}{2}\right\}.
$$
\end{thm}

The proof of Theorem \ref{thm:basic} relies upon the use of the cross-ratio variety $\fX$; this is a 4-dimensional variety parametrising the ${\rm PU}(2,1)$-configuration space of pairwise distinct quadruples on the boundary $\partial\bH^2_\C$ of complex hyperbolic plane. In fact, $\fET$ may be viewed as a 2-dimensional subvariety of $\fX$.

This paper is organised as follows: In Section \ref{sec:prel} we review standard facts about complex hyperbolic plane and the Heisenberg group, as well as about cross-ratio variety. In Section \ref{sec:mod} we prove Theorem \ref{thm:basic} and discuss the particular case of equidistant triples lying in a $\C$-circle in Section \ref{sec:C}.

\medskip

{\it Acknowledgements}. I wish to thank Vassilis Chousionis for suggesting the problem to me, and also Viktor Schroeder for fruitful discussions.

\section{Preliminaries}\label{sec:prel}
The material of this section is standard; a general reference is Goldman's book, \cite{Gol}. In Section \ref{sec:chs} we review complex hyperbolic plane, its boundary and the Heisenberg group. Cartan's angular invariant and complex cross-ratios are in Section \ref{sec:invariants}. Finally, a brief overview of the cross-ratio variety and the ${\rm PU}(2,1)$-configuration of four pairwise distinct points in the boundary of complex hyperbolic plane is found in Section \ref{sec:X-XR}.

\subsection{Complex hyperbolic plane and Heisenberg group}\label{sec:chs}
We consider the vector space $\mathbb{C}^{2,1}$, that is,  $\mathbb{C}^{3}$  with
the Hermitian form of signature $(2,1)$ given by
$$
\left\langle {\bf {z}},{\bf {w}}\right\rangle 
=z_{1}\overline{w}_{3}+z_{2}\overline{w}_{2}+z_{3}\overline{w}_{1}.
$$
We next consider the following subspaces of ${\mathbb C}^{2,1}$:
\begin{equation*}
V_-  =  \Bigl\{{\bf z}\in{\mathbb C}^{2,1}\ :\ 
\langle{\bf z},\,{\bf z} \rangle<0\Bigr\}, \quad
V_0  =  \Bigl\{{\bf z}\in{\mathbb C}^{2,1}\setminus\{{\bf 0}\}\ :\ 
\langle{\bf z},\,{\bf z} \rangle=0\Bigr\}.
\end{equation*}
Denote by  ${\mathbb P}:{\mathbb C}^{2,1}\setminus\{0\}\longrightarrow {\mathbb C}P^2$ 
the canonical projection onto complex projective space. Then the
{\sl complex hyperbolic plane} ${\bf H}_{\mathbb{C}}^{2}$
is defined to be ${\mathbb P}V_-$ and its boundary
$\partial{\bf H}^2_{\mathbb C}$ is ${\mathbb P}V_0$.
Hence  we have
$$
{\bf H}^2_{\mathbb C} = \left\{ (z_1,\,z_2)\in{\mathbb C}^2
\ : \ 2\Re(z_1)+|z_2|^2<0\right\},
$$
and in this manner, ${\bf H}^2_{\mathbb C}$ is the Siegel domain in 
${\mathbb C}^2$.

There are two distinguished points in $V_0$  which we denote by 
${\bf o}$ and $\binfty$:
$$
{\bf o}=\left[\begin{matrix} 0 \\ 0 \\ 1 \end{matrix}\right], \quad
\binfty=\left[\begin{matrix} 1 \\ 0 \\ 0 \end{matrix}\right].
$$
Let ${\mathbb P}{\bf o}=o$ and ${\mathbb P}\binfty=\infty$. 
Then 
$$
\partial{\bf H}^2_{\mathbb C}\setminus\{\infty\} 
=\left\{ (z_1,\,z_2)\in{\mathbb C}^2
\ : \ 2\Re(z_1)+|z_2|^2=0\right\},
$$
and in particular, $o=(0,0)\in\C^2$. 

Conversely, if we are given a point $z=(z_1,z_2)$ of 
${\mathbb C}^2$, then the point 
$$
{\bf z}=\left[\begin{matrix} z_1 \\ z_2 \\ 1 \end{matrix}\right].
$$
is called the {\sl standard lift} of $z$. Therefore the standard lifts of points of the complex hyperbolic plane and its boundary (except the point at infinity) are  vectors of $V_-$ and $V_0$ respectively with the third inhomogeneous coordinate equal to 1.

Complex hyperbolic plane $ {\bf H}_{\mathbb{C}}^{2}$ is a K\"ahler manifold; its K\"ahler structure is given by the Bergman metric. 
The holomorphic sectional curvature  
equals to $-1$ and its real sectional curvature
is pinched between $-1$ and $-1/4$.
The full group of holomorphic isometries 
is the \textsl{projective
unitary group}
$${\rm PU(2,1)}={\rm SU(2,1)}/\{ I,\omega I,\omega^{2}I\},$$
where $\omega$ is a non-real cube root of unity  (that is ${\rm SU}(2,1)$
is a 3-fold covering of ${\rm PU}(2,1)$). 
There are two ways (up to ${\rm PU}(2,1)$ conjugacy) to embed real hyperbolic plane into complex hyperbolic plane; that is, as $\bH^1_\C$ as well as $\bH^2_\R$. These embeddings give rise to  complex lines, i.e., isometric images of the embedding of $\bH^1_\C$ into $\bH^2_\C$ and  Lagrangian planes, i.e., isometric images of $\bH^2_\R$ into $\bH^2_\C$, respectively.

There is an identification of the boundary of the Siegel domain with
the one point compactification of $\C\times\R$: A
 finite point $z$  in the boundary of the Siegel domain has a  standard
lift of the form
$$
{\bf z}
=\left[\begin{matrix} -|z|^2+it \\ \sqrt{2}z\\ 1\end{matrix}\right].
$$
The unipotent stabiliser at infinity acts simply transitively and gives the set of these points the structure of a 2--step nilpotent Lie group, namely the Heisenberg group $\fH$. This is $\C\times\R$ with group law:
$$
(z,t)\star (w,s)=(z+w,t+s+2\Im(z\overline{w})).
$$
The Heisenberg norm (Kor\'anyi gauge) is given by
$$
\left|(z,t)\right|_\fH=\left| \calA(z,t)\right|^{1/2},\quad\text{where}\quad \calA(z,t)=|z|^2-it.
$$
From this norm arises a metric, the Kor\'anyi-Cygan (K-C) metric, on $\fH$ by the relation
$$
d\left((z,t),\,(w,s)\right)
=\left|(z,t)^{-1}\star (w,s)\right|_\fH.
$$
The K-C metric is invariant under 
\begin{enumerate}
 \item [{a)}] left-actions $L_{(w,s)}$ of $\fH$, \; $(z,t)\to(w,s)\star (z,t)$, \; $(w,s)\in\fH$;
%\item the Heisenberg translations $(z,t)\mapsto (z,t+s)$, $s\in\R$;
\item [{b)}] rotations $R_\phi$, \;  $(z,t)\mapsto(ze^{i\phi},t)$, \; $\phi\in\R$;
\end{enumerate}
\begin{enumerate}
\item [{c)}] involution $j$, $j(z,t)=(\overline{z},-t)$.
\end{enumerate} 
These form the  group ${\rm Isom}(\fH,d)$ of {\it Heisenberg isometries}. Note that all the above are orientation-preserving. %The full group ${\rm Isom}^+(\fH,d)$ of Heisenberg isometries comprises elements of ${\rm Isom}^+(\fH,d)$ followed by

The K-C metric is also scaled up to multiplicative constants by the action of  
 \begin{enumerate}
 \item [{d)}] Heisenberg dilations $D_r$, $(z,t)\mapsto$ $(rz,r^2t)$, $r\in\R_+^*.$ 
 \end{enumerate}
and there is also an inversion, defined for each $p=(z,t)\in\fH$, $p\neq o$, by $$(z,t)\mapsto\;\left(\frac{z}{-|z|^2+it},-\frac{t}{|-|z|^2+it|^2|}\right),$$
which satisfies
$$
d_\fH(R(p),R(p'))=\frac{d_\fH(p,p')}{d_\fH(p,o)d_\fH(p',o)}.
$$
The similarity group $G={\rm Sim}(\fH,d)$ comprises compositions of maps of the form a), b), d). Clearly, $G=\fH\times \R\times \R^*_+$.

\subsubsection{$\R$-circles and $\C$-circles}\label{sec:circles}

$\R$-circles are boundaries of Lagrangian planes and $\C$-circles are boundaries of complex lines. They come in two flavours, infinite ones (i.e., containing the point at infinity) and finite ones. We refer to \cite{Gol} for more more details about these curves.

\subsection{Cartan's Angular Invariant }\label{sec:invariants}

Given a triple $(p_1,p_2,p_3)$ of points at the boundary $\partial\bH^2_\C$ the Cartan's angular invariant $\A(p_1,p_2,p_3)$ is defined by
$$
\A(p_1,p_2,p_3)=\arg(-\langle\bp_1,\bp_2\rangle\langle \bp_2,\bp_3\rangle\langle\bp_3,\bp_1\rangle),
$$
where $\bp_i$ are lifts of $p_i$, $i=1,2,3$.
The Cartan's angular invariant lies in $[-\pi/2,\pi/2]$, is independent of the choice of the lifts and remains invariant under the diagonal action of ${\rm PU}(2,1)$. Any other permutation of points produces angular invariants which differ from the above possibly up to sign. The following propositions are in \cite{Gol} to which we also refer the reader for further details:

\begin{prop}
Let $(p_1,p_2,p_3)$ be a triple of points lying in $\partial\bH^2_\C$ and let also $\A=\A(p_1,p_2,$ $p_3)$ be their Cartan's angular invariant. Then:
\begin{enumerate}
\item All points lie in an $\R$-circle if and only if $\A=0$.
\item All points lie in a $\C$-circle if and only if $\A=\pm\pi/2$.
\end{enumerate}
\end{prop}

\begin{prop}
 Suppose that $p_i$ and $p'_i$, $i=1,2,3$, are points in $\partial{\bH_\C^2}$. If there exists a holomorphic isometry $g$ of $\bH^2_\C$ such that $g(p_i)=p'_i$, $i=1,2,3$, then $\A(p_1,p_2,p_3)=\A(p'_1,p'_2,p'_3)$. Conversely, if $\A(p_1,p_2,p_3)=\A(p'_1,p'_2,p'_3)$, then there exists a holomorphic isometry $g$ of $\bH^2_\C$ such that $g(p_i)=p'_i$, $i=1,2,3$. This isometry is unique unless  $p_i$, $i=1,2,3$,  lie in a $\C$-circle. 
\end{prop}

\subsection{Cross-ratio variety and the configuration space}\label{sec:X-XR}
Given a  quadruple of pairwise distinct points $\fp=(p_1,p_2,p_3,p_4)$ in $\partial\bH_\C^2$, we define their complex  cross-ratio as follows:  
$$
\X(p_1,p_2,p_3,p_4)=\frac{\langle \bp_3,\bp_1\rangle \langle \bp_4,\bp_2\rangle}{\langle \bp_4,\bp_1\rangle \langle \bp_3,\bp_2\rangle},
$$
where $\bp_i$ are lifts of $p_i$, $i=1,2,3,4$, see also \cite{KR1}, \cite{PP}, \cite{PP2}. %Note that this definition agrees with the one given in the introduction. 
The cross-ratio is independent of the choice of lifts and remains invariant under the diagonal action of ${\rm PU}(2,1)$. We stress here that for points in the Heisenberg group, the square root of its absolute value is
\begin{equation*}
 |\X(p_1,p_2,p_3,p_4)|^{1/2}=\frac{d_\fH(p_4,p_2)\cdot d_\fH(p_3,p_1)}{d_\fH(p_4,p_1)\cdot d_\fH(p_3,p_2)}.
\end{equation*}

Given a quadruple $\fp=(p_1,p_2,$ $p_3,p_4)$ of pairwise distinct points in the boundary $\partial\bH^2_\C$, all possible permutations of points gives us 24 complex cross-ratios corresponding to $\fp$. Due to  symmetries, see \cite{F}, Falbel showed that all cross-ratios corresponding to a quadruple of points depend on three cross-ratios which satisfy two real equations. Indeed, the following proposition holds; for its proof, see for instance \cite{PP}.

\begin{prop}\label{prop:cross-ratio-equalities}
Let $\fp=(p_1,p_2,p_3,p_4)$ be any quadruple of pairwise distinct points in $\partial \bH^2_\C$. Let
$$
\X_1(\fp)=\X(p_1,p_2,p_3,p_4),\quad \X_2(\fp)=\X(p_1,p_3,p_2,p_4),\quad \X_3(\fp)=\X(p_2,p_3,p_1,p_4).
$$
Then
\begin{eqnarray}\label{eq:cross1}
 &&
|\X_3|^2=|\X_2|^2/|\X_1|^2,\\
&&\label{eq:cross2}
2|\X_1|^2\Re(\X_3)=|\X_1|^2+|\X_2|^2-2\Re(\X_1)-2\Re(\X_2)+1.
\end{eqnarray}
\end{prop}
Equations (\ref{eq:cross1}) and  (\ref{eq:cross2}) define a 4-dimensional real subvariety of $\C^3$ which we call the {\it cross-ratio variety} $\fX$. This variety is isomorphic to the subset $\mathfrak{F'}$ of the ${\rm PU}(2,1)$ configuration space $\mathfrak{F}$ of pairwise disjoint quadruples of points in $\partial\bH^2_\C$, comprising quadruples whose points do not all lie in the same $\C$-circle. In the latter case, we have a 2--1 map between the subset $\mathfrak{F}_\R$ comprising of quadruples whose points all lie in a $\C$-circle and the subvariety $\fX_\R$ of $\fX$ defined by
$$
\fX_\R=\{(\X_1,\X_2)\in\R^2_{*}\;|\;\X_1+\X_2=1\},
$$
see \cite{F}, \cite{CG}.
For further reference, we shall need the following:
\begin{rem}\label{rem:AX}
Let $\fp$ a quadruple of pairwise distinct points in $\partial\bH^2_\C$ which do not all lie in the same $\C$-circle and let also $\X_i(\fp)$, $i=1,2,3$ be as above. Let also $a=\arg(\X_1)$, $b=\arg(\X_2)$, $c=\arg(\X_3)$. We have
$$
a=\A_1-\A_2,\quad b=-\A_2-\A_4,\quad c=\A_4-\A_1.
$$
Here,
$$
\A_1=\A(p_2,p_3,p_4),\quad \A_2=\A(p_1,p_3,p_4),\quad \A_4=\A(p_1,p_2,p_3).
$$
As for $\A_3=\A(p_1,p_2,p_4)$ we have
$$
\A_3+\A_1=\A_2+\A_4.
$$
\end{rem}

\section{The Configuration Space of Equidistant Triples}\label{sec:mod}

In this section we are going to prove Theorem \ref{thm:basic}. The proof will follow after a series of lemmas which follow below.
\subsection{The lemmas} \label{sec:lemmas}
Throughout this section we will have the following notation: We shall denote by $P$ a triple $(p_1,p_2,p_3)$ of pairwise distinct points in the Heisenberg group $\fH$. We will consider also the quadruple $\fp=(p_1,\infty,p_2,p_3)$; %and we shall assume that
%\medskip
%\center{\it the points in $\fp$ do not all lie in a $\C$-circle}.
%\medskip
let $\X_i(\fp)$, $i=1,2,3$ be the corresponding point on the cross-ratio variety $\fX$ and let  $$a=\arg(\X_1(\fp)),\quad b=\arg(\X_2(\fp)),\quad c=\arg(\X_3(\fp)).$$
There is an important note here: points of $\fp$ {\it cannot} all lie in the same (infinite) $\C$-circle. To see this, normalise so that 
$$
p_1=(0,0),\quad p_2=(0,t),\quad p_3=(0,s),\;t,s\in\R.
$$ 
Then conditions $d(p_1,p_2)=d(p_1,p_3)=d(p_2,p_3)$ deduce
$$
|t|^{1/2}=|s|^{1/2}=|t-s|^{1/2},
$$
which cannot happen.

\begin{lem}\label{lem:1}
The triple $P=(p_1,p_2,p_3)$ is  equidistant if and only if
$$
|\X_1(\fp)|=|\X_2(\fp)|=1. %|\X_3(\fp)|=1.
$$
\end{lem}
\begin{proof}
Since  $P=(p_1,p_2,p_3)$ is an equidistant triple, we have $$
d(p_1,p_2)=d(p_2,p_3)=d(p_3,p_1),
$$
where $d$ is the Kor\'anyi distance. The result follows from the formulae
$$
|\X_1(\fp)|^2=\frac{d(p_2,p_1)}{d(p_3,p_1)},\quad |\X_2(\fp)|^2=\frac{d(p_3,p_2)}{d(p_3,p_1)}.%, \quad |\X_3(\fp)|=\frac{|\X_2(\fp)|}{|\X_1(\fp)|}.
$$
Note that the above imply as well $|\X_3(\fp)|=1$.
\end{proof}

\begin{lem}\label{lem:2}
If $P=(p_1,p_2,p_3)$ is an equidistant triple then $a,b,c$ satisfy
\begin{equation}\label{eq:modeq}
\cos a+\cos b+\cos c=\frac{3}{2}.
\end{equation}
\end{lem}
\begin{proof}
Consider the cross-ratio variety equations (\ref{eq:cross1}) and (\ref{eq:cross2}) as in the previous section.
%\begin{eqnarray}
%&&\label{eq:cr1}
%|\X_2|=|\X_1||\X_3|,\\
%&&\label{eq:cr2}
%2|\X_1|^2\Re(\X_3)=|\X_1|^2+|\X_2|^2-2\Re(\X_1)-2\Re(\X_2)+1.
%\end{eqnarray}
We may rewrite (\ref{eq:cross2}) equivalently as
\begin{equation}\label{eq:cr2.1}
2|\X_1||\X_2|\cos c=|\X_1|^2+|\X_2|^2-2|\X_1|\cos a-2|\X_2|\cos b +1.
\end{equation}
If $P$ is an equidistant triple then $|\X_i(\fp)|=1$, $i=1,2,3$ and thus (\ref{eq:cr2.1}) reduces to (\ref{eq:modeq}).
\end{proof}

\begin{lem}\label{lem:2.1}
If $P=(p_1,p_2,p_3)$ is an equidistant triple,  $\fp=(p_1,\infty,p_2,p_3)$ and $(a,b,c)$ as above. If $g\in G={\rm Sim}(\fH)$, we set
$$
P'=g(P)=(g(p_1),g(p_2),g(p_3))=(p_1',p_2',p_3'),\quad \fp'=(p_1',\infty,p_2',p_3').
$$
Then
$$
a'=\arg(\X_1(\fp'))=a,\quad b'=\arg(\X_2(\fp'))=b,\quad c'=\arg(\X_3(\fp'))=c.
$$
\end{lem}
\begin{proof}
The proof is immediate by invariance of cross-ratios.
\end{proof}

\begin{lem}\label{lem:3}
Let $(a,b,c)\in[-2\pi/3,2\pi/3]^3$ such that it satisfies 
$$
\cos a+\cos b+\cos c=\frac{3}{2}.
$$
 Then there exists an equidistant triple $P=(p_1,p_2,p_3)$ such that if $\fp=(p_1,\infty,p_2,p_3)$ then 
$$
\arg(\X_1(\fp))=a,\quad\arg(\X_2(\fp))=b,\quad\arg(\X_3(\fp))=c.
$$
\end{lem}

\begin{proof}
Set $2\eta=\arg(1-e^{ia}-e^{ib})$. We have
\begin{eqnarray*}
|1-e^{ia}-e^{ib}|^2&=&3-2\cos a-2\cos b+2\cos(a-b)\\
&=&3-4\cos\left(\frac{a+b}{2}\right)\cos\left(\frac{a-b}{2}\right)+4\cos^2\left(\frac{a-b}{2}\right)-2\\
&=&4\left(\cos\left(\frac{a-b}{2}\right)-\frac{1}{2}\cos\left(\frac{a+b}{2}\right)\right)^2+\sin^2\left(\frac{a+b}{2}\right).
%&=& 2\cos c+2\cos(a-b)\\
%&=&4\cos(A_1)\cos(A_4),
\end{eqnarray*}
This is strictly positive unless
$$
(a,b,c)\in B=\{(\pi/3, -\pi/3,\pm\pi/3),(-\pi/3, \pi/3,\pm\pi/3)\}. %\quad %n\in\mathbb{Z}.
$$
Assume first that $(a,b,c)\notin B$ and set
$$
A_1=%\A(\infty,p_2,p_3)
\frac{a-b-c}{2},%\quad\A_2=\frac{-a-b-c}{2},\quad \A_3=\frac{-a-b+c}{2},
\quad A_4=%\A(p_1,\infty,p_2)=
\frac{a-b+c}{2},
$$
with $A_1,A_4\in(-\pi/2,\pi/2)$.
Notice that we have
\begin{eqnarray*}
4\cos(A_1)\cos(A_4)&=&2\cos(a-b)+2\cos c\\
&=&2\cos(a-b)+3-2\cos a-2\cos b\\
&=&|1-e^{ia}-e^{ib}|^2>0.
\end{eqnarray*}
Consider the triple $P=(p_1,p_2,p_3)$ of points in $\fH$ where
$$
p_1=\left(\sqrt{\cos(A_4)}e^{i\left(\frac{b-c}{2}-\eta\right)},\;\sin(A_4)\right),\quad p_2=(0,0),\quad p_3=\left(-\sqrt{\cos(A_1)}e^{i\left(\eta-\frac{a}{2}\right)},\;\sin(A_4)\right),
$$
and the quadruple $\fp=(p_1,\infty,p_2,p_3)$, with lifts:
\begin{eqnarray*}
&&
\bp_1=\left[\begin{matrix}
-e^{-ia/2}\\
\\
\sqrt{2\cos(A_4)}e^{-i\eta}\\
\\
e^{i(c-b)/2}
\end{matrix}\right],\quad
\infty=\left[\begin{matrix}
1\\
0\\
0
\end{matrix}\right],\quad
\bp_2=\left[\begin{matrix}
0\\
0\\
1
\end{matrix}\right],\quad
\bp_3=\left[\begin{matrix}
e^{i(b+c)/2}\\
\\
\sqrt{2\cos(A_1)}e^{i\eta}\\
\\
-e^{ia/2}
\end{matrix}\right].
\end{eqnarray*}
Now,
\begin{eqnarray*}
&&
\langle \bp_1,\infty\rangle=e^{i(c-b)/2},\quad \langle \bp_1,\bp_2\rangle=-e^{-ia/2},%\\
%&&
\quad\langle\infty,\bp_2\rangle=1,\\
&&
%\quad
 \langle\infty,\bp_3\rangle=-e^{-ia/2},%\\
%&& 
\quad\langle \bp_2,\bp_3\rangle=e^{-i(b+c)/2},
\end{eqnarray*}
and
\begin{eqnarray*}
\langle \bp_1,\bp_3\rangle&=& e^{-ia}+e^{-ib}+\sqrt{2\cos(A_1)\cos(A_4)}\cdot e^{-2i\eta}\\
&=&e^{-ia}+e^{-ib}+|1-e^{-ia}-e^{-ib}|\cdot\frac{1-e^{-ia}-e^{-ib}}{|1-e^{-ia}-e^{-ib}|}\\
&=&1.
\end{eqnarray*}
This gives
$$
\X_1=e^{ia},\quad \X_2=e^{ib},\quad \X_3=e^{ic},
$$
which proves our claim. 

\medskip

 Finally, we consider $(a,b,c)\in B$; we shall only treat the case where $a=\pi/3$, $b=-\pi/3$ and $c=\pi/3$, in other words when $\A_4=\pi/2$ and $\A_1=\pi/6$. Then we set
\begin{eqnarray*}
&&
\bp_1=\left[\begin{matrix}
\frac{-\sqrt{3}+i}{2}\\
\\
0\\
\\
\frac{1+i\sqrt{3}}{2}
\end{matrix}\right],\quad
\infty=\left[\begin{matrix}
1\\
0\\
0
\end{matrix}\right],\quad
\bp_2=\left[\begin{matrix}
0\\
0\\
1
\end{matrix}\right],\quad
\bp_3=\left[\begin{matrix}
1\\
\\
3^{1/4}\\
\\
-\frac{\sqrt{3}+i}{2}
\end{matrix}\right].
\end{eqnarray*}
All other cases can be treated in a similar manner.
\end{proof}

\begin{lem}\label{lem:4}
Let $(a,b,c)\in\calE$ and consider from Lemma \ref{lem:3} the equidistant triple  $P=(p_1,p_2,p_3)$ of points in $\fH$ which is such that 
$$
a=\arg(\X_1(\fp)),\quad b=\arg(\X_2(\fp)),\quad c=\arg(\X_3(\fp)).
$$ 
Then %if the points of $\fp=(p_1,\infty,p_2,p_3)$ do not all lie in the same $\C$-circle, 
any other equidistant triple $P'=(p_1',p_2',p_3')$ such that 
$$
a=\arg(\X_1(\fp')),\quad b=\arg(\X_2(\fp')),\quad c=\arg(\X_3(\fp'))
$$  
is of the form
$
P'=g(P)
$, where $g\in G={\rm Sim}(\fH)$. %If all points of $\fp$ lie in the same $\C$-circle, any other equidistant triple $P'=(p_1',p_2',p_3')$ with $a,b$ and $c$ as above is obtained either by a $g\in G$ or by a $g\in G$ followed by conjugation $j$. 
\end{lem}
\begin{proof}
We consider $P$, $P'$ and $\fp$, $\fp'$, respectively. Since $\X_i(\fp)=\X_i(\fp')$, $i=1,2,3$, we have from Proposition 5.10 in \cite{PP} and Lemma 5.5 in \cite{F} that since not all points in $\fp$ lie in the same $\C$-circle, there exists a $g\in{\rm PU}(2,1)$  such that $g(p_i)=p_i'$, $i=1,2,3$ and $g(\infty) =\infty$.
%\begin{enumerate}
%\item In case where not all points in $\fp$ lie in the same $\C$-circle, there exists a $g\in{\rm PU}(2,1)$  such that $g(p_i)=p_i'$, $i=1,2,3$ and $g(\infty) =\infty$.
%\item In case where all points in $\fp$ lie in the same $\C$-circle, then  there exists either a $g\in{\rm PU}(2,1)$ or a $g$ followe by conjugation $j$ such that $g(p_i)=p_i'$, $i=1,2,3$ and $g(\infty) =\infty$.  
%\end{enumerate}
That is, %in all cases 
$g\in G$ and the proof is complete.
\end{proof}

\subsection{The Equidistant surface}
Equation (\ref{eq:modeq}) is the equation of a hypersurface in $\R^3$ which we shall call {\it equidistant hypersurface} and denote it by $\calE$. This hypersurface comprises of infinitely many connected components. Notice that we have
$$
(a,b,c)\in\calE\implies (a+2n_1\pi,b+2n_2\pi,c+2n_3\pi)\in\calE,\quad n_1,n_2,n_3\in\mathbb{Z}.
$$ 
On the other hand, since for instance
$$
\cos a=\frac{3}{2}-\cos b-\cos c\ge -\frac{1}{2},
$$
we have that $\cos a$, and in the same manner $\cos b$ and $\cos c$, are $\ge -1/2$. We deduce that the connected components of $\calE$ may be taken by transporting the central component where
$$
(a,b,c)\in[-2\pi/3,2\pi/3]^3,
$$
by multiples of $2\pi$ in all possible directions, see the figure where the central component of $\calE$  is clearly shown.

\begin{figure}
\title{Central component of equidistant surface}
	\centering
	\includegraphics[width=\linewidth]{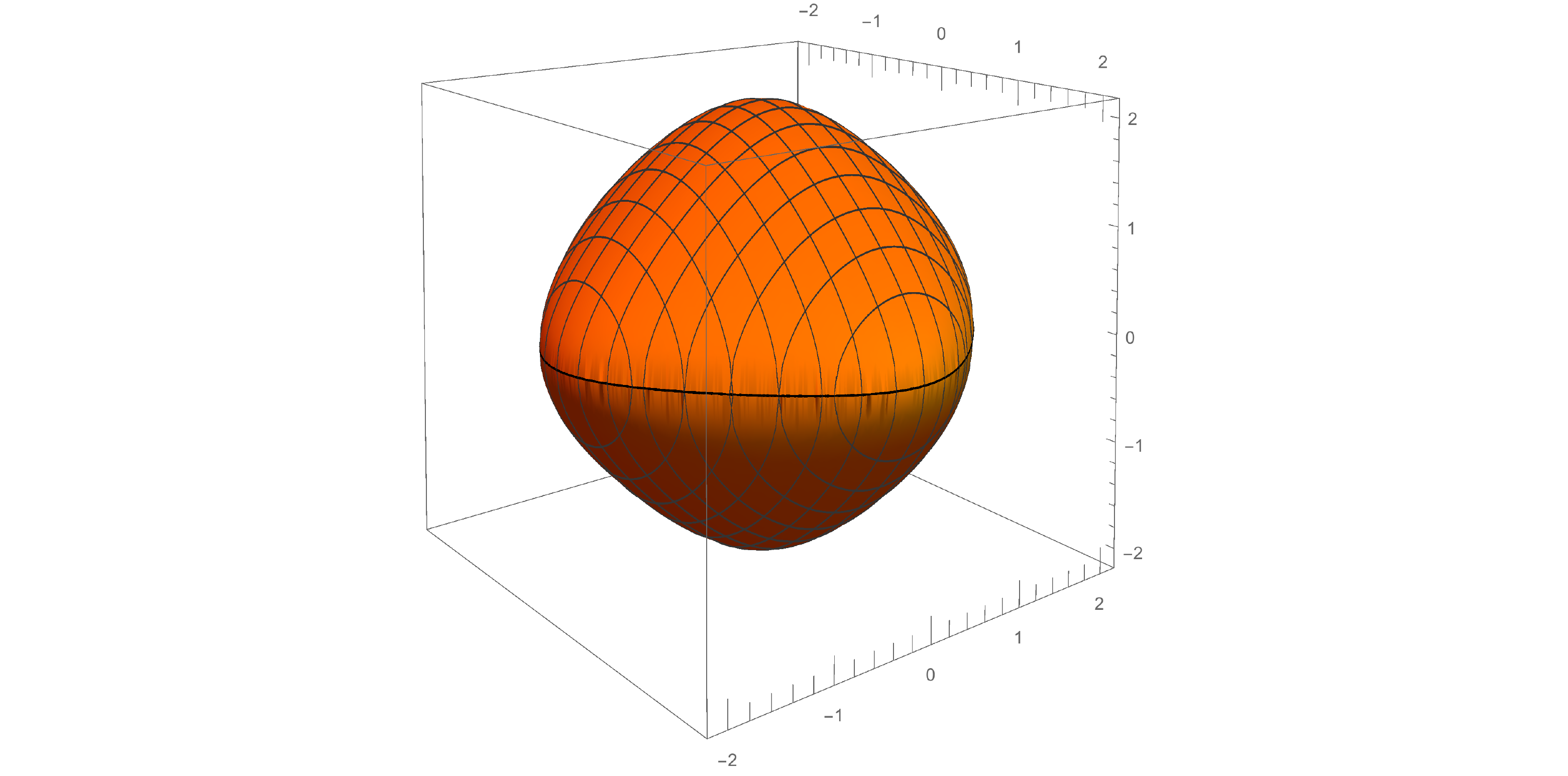}
\end{figure}

\subsection{Proof of Theorem \ref{thm:basic}}
Given an equidistant triple $P=(p_1,p_2,p_3)$ we consider the quadruple $\fp=(p_1,\infty,p_2,p_3)$ in the ${\rm PU}(2,1)$-configuration space of pairwise distinct points on the boundary $\partial\bH^2_\C$ of complex hyperbolic plane. Let $\X_i(\fp)$, $i=1,2,3$ be the cross-ratios associated to $\fp$; the triple $(\X_1(\fp),\X_2(\fp),\X_3(\fp))$ defines a point in the cross=ratio variety $\fX$.  In particular, in this case we have by Lemma \ref{lem:1} that $|\X_i(\fp)|=1$, $i=1,2,3$ and moreover, if
$$
a=\arg(\X_1(\fp)),\quad b=\arg(\X_1(\fp)),\quad c=\arg(\X_3(\fp)),
$$ 
then from Lemma \ref{lem:2}
$$
\cos a+\cos b+\cos c=\frac{3}{2}.
$$
By Lemma \ref{lem:2.1} this equation is invariant by the diagonal action of $G$ on the space of equidistant triples $\calET$:
This proves that the map
$$
\fET\to\calE,\quad [P]\mapsto(a,b,c),
$$
is well-defined.

Conversely, if $(a,b,c)\in\calE$ where $(a,b,c)\in[-2\pi/3,2\pi/3]$, by Lemma \ref{lem:3} there exists a $P=(p_1,p_2,p_3)\in\calET$ such that if $\fp=(p_1,\infty,p_2,p_3)$ then $a=\arg(\X_1(\fp))$, $b=\arg(\X_2(\fp))$ and $c=\arg(\X_3(\fp))$; therefore $\fET\to\calE$ is onto. Finally, %if $P'$ is another equidistant triple such that $[P']\mapsto(a,b,c)$ then 
by Lemma \ref{lem:4} %there exists a similarity $g\in G$ such that $P'=g(P)$, or a similarity followed by conjugation in the case where all points of $\fp$ lie in the same $\C$-circle. This proves that
 $\fET\to\calE$ is 1--1 when the points in $\fp$ do not all lie in the same $\C$-circle and 2--1 when all points in $\fp$ lie in the same $\C$-circle. This concludes the proof of Theorem \ref{thm:basic}.\qed

\subsection{The $\C$-circle case}\label{sec:C}
The subset $\fET_\C$ of $\fET$ comprising equivalent equidistant triples of points lying on a $\C$-circle is of special interest. We can show in an elementary way that $\fET_\C$ is just two points on the equidistant hypersurface $\calE$. We start with a lemma:
\begin{lem}\label{lem:A-X}
With the assumptions of Section \ref{sec:lemmas} let $P=(p_1,p_2,p_3)$ and $\fp=(p_1,\infty,p_2,p_3)$. Then
\begin{equation*}
a+b+c=-2\A_2=-2\A(p_1,p_2,p_3)=-2\A.
\end{equation*} 
\end{lem}
\begin{proof}
We have
$$
a+b+c=\A_1-\A_2-\A_2-\A_4+\A_4-\A_1=-2\A_2.
$$
%\begin{eqnarray*}
%a+b+c&=&\arg\left(\frac{\langle\bp_2, \bp_1\rangle\langle\bp_3, \bp_2\rangle\langle\bp_3, \bp_2\rangle}{\langle\bp_3, \bp_1\rangle\langle\bp_3, \bp_1\rangle\langle\bp_1, \bp_2\rangle}\right)\\
%&=&\arg\left(\frac{\langle\bp_3, \bp_2\rangle\langle\bp_2, \bp_1\rangle\langle\bp_1, \bp_3\rangle\langle\bp_3, \bp_2\rangle}{\langle\bp_3, \bp_1\rangle\langle\bp_1, \bp_3\rangle\langle\bp_3, \bp_1\rangle\langle\bp_1, \bp_2\rangle}\right)\\
%&=&2\arg\left(-\langle\bp_3, \bp_2\rangle\langle\bp_2, \bp_1\rangle\langle\bp_1, \bp_3\rangle\right)\\
%&=&2\A(p_3,p_2,p_1)=-2\A.
%\end{eqnarray*}
\end{proof}

We now prove
\begin{prop}\label{prop:Cequid}
The subset $\fET_\C$ of $\fET$ comprising equivalence classes of equidistant triples of points lying on the same $\C$-circle is on bijection with the points $(\pi/3,\pi/3,\pi/3)$ or $(-\pi/3,-\pi/3,-\pi/3)$ of the equidistant surface $\calE$.
\end{prop}
\begin{proof}
We will show that three equidistant points lie on a $\C$-circle if and only if the equidistant hypersurface  reduces to the point $(\pi/3,\pi/3,\pi/3)$ or $(-\pi/3,-\pi/3,-\pi/3)$.
We start by assuming that the three points lie on a $\C$-circle, that is, $\A=\pm\pi/2$. Here, $\A=\A(p_1,p_2,p_3)$. Since from Lemma \ref{lem:A-X} we have
%\begin{equation}\label{eq:mod1}
$$
a+b+c=-2\A,
$$
%\end{equation} 
Equation (\ref{eq:modeq}) becomes
$$
\cos a+\cos b-\cos(a+b)=\frac{3}{2}.
$$
This is written equivalently as
$$
2(\cos a+\cos b)-2\cos a\cos b+2\sin a \sin b=\cos^2a+\sin^2a
+\cos^2b+\sin^2b+1,
$$
or,
$$
(\cos a+\cos b)^2-2(\cos a+\cos b)+1+(\sin a-\sin b)^2=0,
$$
that is,
$$
(\cos a+\cos b-1)^2+(\sin a-\sin b)^2=0.
$$
Therefore we obtain,
$$
\cos a+\cos b=1\quad\text{and}\quad \sin a=\sin b.
$$
This gives
$$
a=b=c=\pm\frac{\pi}{3}.
$$
Conversely, suppose that three points lie on a $\C$-circle and $\arg(\X_i)=\pm\frac{\pi}{3}$. Then $p_i$ are equidistant. Indeed, from equation (\ref{eq:cr2.1}) we have
$$
|\X_1||\X_2|=|\X_1|^2+|\X_2|^2-|\X_1|-|\X_2|+1.
$$
Factoring out we may write equivalently
$$
\left(|\X_1|-\frac{|\X_2|}{2}-\frac{1}{2}\right)^2+\frac{3}{4}(|\X_2|-1)^2=0.
$$
This gives $|\X_1|=|\X_2|=1$ and therefore the points are equidistant.
\end{proof}

\end{document}